\documentclass{article}
\usepackage{amsmath}
\usepackage{amsfonts}
\usepackage{amsthm}
\usepackage{amssymb}
\usepackage{enumerate}

\newtheorem{thm}{Theorem}
\newtheorem{prop}[thm]{Proposition}
\newtheorem{cor}[thm]{Corollary}
\newtheorem{lem}[thm]{Lemma}

\theoremstyle{definition}
\newtheorem{defn}[thm]{Definition}

\begin{document}

\title{Moments of Partitions and Derivatives of Higher Order}

\author{Shaul Zemel}

\maketitle

\begin{abstract}
We define moments of partitions of integers, and show that they appear in higher order derivatives of certain combinations of functions.
\end{abstract}

\section*{Introduction and Statement of the Main Result}

Changes of coordinates grew, through the history of mathematics, from a powerful computational tool to the underlying object behind the modern definition of many objects in various branches of mathematics, like differentiable manifolds or Riemann surfaces. With the change of coordinates, all the objects that depend on these coordinates change their form, and one would like to investigate their behavior. For functions of one variable, like holomorphic functions on Riemann surfaces, this is very easy, but one may ask what happens to the derivatives of functions under this operation. The answer is described by the well-known formula of Fa\`{a} di Bruno for the derivative of any order of a composite function. For the history of this formula, as well as a discussion of the relevant references, see \cite{[J]}.

For phrasing Fa\`{a} di Bruno's formula, we recall that a partition $\lambda$ of some integer $n$, denoted by $\lambda \vdash n$, is defined to be a finite sequence of positive integers, say $a_{l}$ with $1 \leq l \leq L$, written in decreasing order, whose sum is $n$. The number $L$ is called the \emph{length} of $\lambda$ and is denoted by $\ell(\lambda)$, and given a partition $\lambda$, the number $n$ for which $\lambda \vdash n$ is denoted by $|\lambda|$. Another method for representing partitions, which will be more useful for our purposes, is by the \emph{multiplicities} $m_{i}$ with $i\geq1$, which are defined by $m_{i}=\big|\;\{1 \leq l \leq L|a_{l}=i\}\big|$, with $m_{i}\geq0$ for every $i\geq1$ and such that only finitely many multiplicities are non-zero. In this case we have $|\lambda|=\sum_{i\geq1}im_{i}$ and $\ell(\lambda)=\sum_{i\geq1}m_{i}$. Note that the empty partition, in which all the multiplicities $m_{i}$ vanish, is allowed. It is considered to be partition of 0, with length 0.

Therefore when some partition $\lambda$ is known from the context, the numbers $m_{i}$ will denote the associated multiplicities, and in case several partitions are involved we may write $m_{i}(\lambda)$ for clarification. Assume that $f$ is a function of $z$ and the variable $z$ is a function of another variable $t$, say $z=\varphi(t)$, and we wish to differentiate the resulting function of $t$ successively. The formula of Fa\`{a} di Bruno is the answer to this question, which we can write explicitly as
\begin{equation}
(f\circ\varphi)^{(n)}(t)=\frac{d^{n}}{dt^{n}}\big((f(\varphi(t)\big)=\sum_{\lambda \vdash n}\frac{n!}{\prod_{i=1}^{n}(i!)^{m_{i}}m_{i}!}f^{(\ell(\lambda))}\big(\varphi(t)\big)\prod_{i=1}^{n}\big(\varphi^{(i)}(t)\big)^{m_{i}}. \label{FaadiBruno}
\end{equation}
We remark that gathering these formulae for all $n$ together, and noticing that $\lambda$ appears in the derivative of order $|\lambda|$, yields a structure of a Hopf algebra on the polynomial ring of infinitely many variables, graded appropriately---see, e.g., \cite{[FGV]}.

\smallskip

Equation \eqref{FaadiBruno} can be viewed as describing the behavior of derivatives of functions on 1-dimensional objects (like Riemann surfaces, when the variables are locally taken from $\mathbb{C}$) under changing the coordinate. However, functions are not the only type of forms that can be defined on 1-dimensional objects, and the next forms to consider are differentials, and more generally $q$-differentials. These are defined such that their coordinate changes also involve the $q$th power of the derivative of the coordinate change, namely if a $q$-differential is expressed in a coordinate neighborhood as $f(z)$ times the formal symbol $(dz)^{q}$, then when we change the coordinate via $z=\varphi(t)$ the description in the coordinate $t$ is $f\big(\varphi(t)\big)\varphi'(t)^{q}$ times $(dt)^{q}$ (see, e.g., Section III.4.12 of \cite{[FK]}). While simply differentiating such expressions may seem a bit unnatural, this operation does appear, for example, in the proof of Proposition III.5.10 of \cite{[FK]}, which states that if $d$ is the dimension of the space of $q$-differentials on a Riemann surface $X$ then the Wronskian of this space is an $m$-differential, where $m=\frac{d}{2}(d+2q-1)$. While the proof of the latter statement takes only the ``essential terms'' of this derivative, where no combinatorial calculations have to be carried out, it does leave open the question about the formula for the $n$th derivative of such a transformation rule, and whether some interesting combinatorial phenomena hide in it. The dependence on $q$ as a number becomes formal, and the expression that we investigate in this manner is the $n$th derivative of an expression like $f\big(\varphi(t)\big)g\big(\varphi'(t)\big)$, or just $(f\circ\varphi)\cdot(g\circ\varphi')$ when we omit the variable $t$. In fact, $g$ needs not be composed with the first derivative of $\varphi$, but can rather be composed with the derivative $\varphi^{(s)}$ of any order $s\geq0$.

\smallskip

The question that we tackle in this paper is therefore finding an explicit formula for the $n$th derivative of the expression $(f\circ\varphi)\cdot\big(g\circ\varphi^{(s)}\big)$, in terms of the derivatives of $f$, $g$, and $\varphi$. The fact that the formula, which is given in Equation \eqref{finres} below, involves partitions, is, of course, no big surprise. But in addition to the combinatorial coefficients appearing in Fa\`{a} di Bruno's formula from Equation \eqref{FaadiBruno}, the resulting coefficient involves some numbers that we call \emph{moments} of partitions. More precisely, given an integer $k\geq1$ and a partition $\lambda$, with the summands $a_{l}$, $1 \leq l\leq\ell(\lambda)$ and the multiplicities $m_{i}$, we define its \emph{$k$th moment} to be \[p_{k}(\lambda)=\sum_{l=1}^{\ell(\lambda)}a_{l}^{k}=\sum_{i\geq1}i^{k}m_{i}.\] In particular the first moment of $\lambda$ is just $|\lambda|$ by definition. The notation $p_{k}$ comes from the theory of symmetric functions, as this moment is the value attained by the $k$th power sum function on the numbers $a_{l}$, $1 \leq l\leq\ell(\lambda)$. However, there are several natural bases for the ring of symmetric functions, and in particular one can take the basis arising from the \emph{elementary symmetric functions} $\{e_{r}\}_{r=0}^{\infty}$, which appear, e.g., in the expressions for the coefficients of a polynomial in terms of its roots. We shall therefore denote by $e_{r}(\lambda)$ the \emph{$r$th elementary moment} of $\lambda$, which is obtained by substituting the $a_{l}$s into the $r$th elementary symmetric function $e_{r}$. Note that every symmetric function with index 0 is the constant 1, so that the 0th moment of every partition is 1 (even though the formula for $p_{k}$ above would give $\ell(\lambda)$ when $k=0$).

An interesting feature of the resulting formula is that for expressing the coefficient associated with $\lambda$, we first have to modify $\lambda$ in two different directions, and take the elementary moments of this modification. More explicitly, given an integer $s\geq0$ and a partition $\lambda$, we shall denote by $\lambda^{>s}$ the \emph{$s$th truncation} of $\lambda$, which is obtained by eliminating any number $a_{l}$ which satisfies $a_{l} \leq s$, or equivalently by setting each $m_{i}$ with $i \leq s$ to 0 and leaving the multiplicities $m_{i}$ with $i>s$ at their value $m_{i}(\lambda)$. Note that this operation may transform some non-trivial partitions into the trivial one, all the moments of positive indices of which vanish by definition. In addition, for every partition $\mu$ and integer $s\geq0$ we denote by $(\mu)_{s}$ the partition obtained by replacing each number $a_{l}$ by its \emph{Pochhammer symbol} $(a_{l})_{s}=\prod_{\upsilon=0}^{s-1}(a_{l}-\upsilon)=\frac{a_{l}!}{(a_{l}-s)!}$ (the latter equality holding also when $0 \leq a_{l}<s$, since then the numerator is finite and the denominator is infinite, but we shall use it for $\mu=\lambda^{>s}$ where no such indices appear). Using this notation, our main result states that the $n$th derivative of $(f\circ\varphi)\cdot\big(g\circ\varphi^{(s)}\big)$ is given by
\begin{equation}
\sum_{r=0}^{n}\sum_{\substack{\lambda \vdash n+rs \\ \ell(\lambda^{>s}) \geq r}}\frac{n!e_{r}\big((\lambda^{>s})_{s}\big)}{\prod_{i=1}^{n}(i!)^{m_{i}}m_{i}!}\big(f^{(\ell(\lambda)-r)}\circ\varphi\big)\big(g^{(r)}\circ\varphi^{(s)}\big)\prod_{i=1}^{n}\big(\varphi^{(i)}\big)^{m_{i}}, \label{finres}
\end{equation}
where $m_{i}=m_{i}(\lambda)$ are the multiplicities associated with the partition $\lambda$. An immediate corollary is that the coefficients from Equation \eqref{finres} are integers, a fact that is much less obvious than the integrality of the coefficients from Equation \eqref{FaadiBruno}, which have combinatorial interpretations (these are also the coefficients appearing in the summands with $r=0$ in Equation \eqref{finres}).

\smallskip

It is well-known (see, e.g., Subsection 3.2.5 of \cite{[MS]}) that the formula of Fa\`{a} di Bruno from Equation \eqref{FaadiBruno} is related to certain multi-variate polynomials, called the \emph{Bell polynomials}, and appropriate substitutions in them produce the Touchard polynomials, the coefficients of which are the Stirling numbers of the second kind. We therefore define modified Bell polynomials using Equation \eqref{finres}, and establish a few of their properties as well as of those of the associated modified Stirling numbers of the second kind.

\smallskip

The rest of the paper is divided into 4 sections. Section \ref{FdBSec} presents a (well-known) proof of the formula of Fa\`{a} di Bruno's from Equation \eqref{FaadiBruno}, the ideas of which will be later used for proving the main result. Section \ref{Propofer} establishes some properties of the elementary symmetric functions that we shall need, and then Section \ref{Main} proves Equation \eqref{finres} and deduces some consequences. Finally, Section \ref{Bell} defines the modified Bell polynomials and the modified Stirling numbers of the second kind and proves their properties.

\section{A Proof of Fa\`{a} di Bruno's Formula \label{FdBSec}}

We will prove our main result by induction on $n$, like one of the many proofs of Fa\`{a} di Bruno's formula, which we shall give here in Proposition \ref{FdBprop}. The reason we include this proof is for introducing the tools that we shall use for the main result below.

\smallskip

First we introduce a notation that will help us avoid undefined terms. Recall the Kronecker's $\delta$-symbol $\delta_{x,y}$, which is defined to be 1 when $x=y$ and 0 otherwise. Following our previous paper \cite{[Z]}, we shall use the complementary symbol $\overline{\delta}_{x,y}=1-\delta_{x,y}$, which equals 0 when $x=y$ and 1 otherwise. In addition, we make the following definition of partitions.
\begin{defn}
Let an integer $j\geq1$ and a partition $\lambda$ be given, and assume that $m_{j}(\lambda)\geq1$. We denote by $\lambda-\varepsilon_{j}$ the partition obtained by omitting one of the instances of $j$ (i.e., by deleting one of the numbers $a_{l}$ which equals $j$, and re-indexing). We also write $\lambda_{j}$ for the partition obtained by subtracting 1 from one of the numbers $a_{l}$ that equal $j$, and then deleting trivial terms and again re-indexing in decreasing order. \label{lambdajdef}
\end{defn}

The partitions from Definition \ref{lambdajdef} have the parameters given in the following simple lemma.
\begin{lem}
For the partition $\lambda-\varepsilon_{j}$ we have
\[|\lambda-\varepsilon_{j}|=|\lambda|-j,\ \ell(\lambda-\varepsilon_{j})=\ell(\lambda)-1,\mathrm{\ and\ }m_{i}(\lambda-\varepsilon_{j})=m_{i}-\delta_{i,j}\mathrm{\ for\ }i\geq1.\] On the other hand, for $\lambda_{j}$ we get \[|\lambda_{j}|=|\lambda|-1,\ \ell(\lambda_{j})=\ell(\lambda)-\delta_{j,1},\mathrm{\ and\ }m_{i}(\lambda_{j})=\begin{cases}m_{i}(\lambda)-1 & i=j \\ m_{i}(\lambda)+1 & i=j-1\geq1 \\ m_{i}(\lambda) & \mathrm{otherwise}. \end{cases}\] \label{lambdaj}
\end{lem}
Note that when $j=1$ we have $\lambda-\varepsilon_{1}=\lambda_{1}$ in Definition \ref{lambdajdef}, and the two lines from Lemma \ref{lambdaj} coincide. In addition, we shall make the convention, which is appropriate by our definition, that $m_{0}(\lambda)=0$ for every partition $\lambda$ (this is why we wrote $i=j-1\geq1$ in the second case in Lemma \ref{lambdaj}---we do not want $m_{0}(\lambda_{1})$ to be 1). This will be convenient for many statements below.

\smallskip

We begin with establishing the combinatorial identity behind one of the proofs of Fa\`{a} di Bruno's formula. While this proof is well-known, it contains the ideas that will be used later for proving the main result as well. We denote the combinatorial coefficient from Equation \eqref{finres} by $C_{\lambda,r}^{(s)}$, so that the one from Equation \eqref{FaadiBruno} is $C_{\lambda,0}^{(s)}$, regardless of the value of $s$.
\begin{lem}
If $\lambda$ is any partition such that $|\lambda|>0$, with multiplicities $\{m_{i}\}_{i\geq1}$, then the coefficient $C_{\lambda,0}^{(s)}$ can be written as $\sum_{j\geq1}(m_{j-1}+1)\overline{\delta}_{m_{j},0}C_{\lambda_{j},0}^{(s)}$. \label{recr0}
\end{lem}

\begin{proof}
Lemma \ref{lambdaj} shows that for every $j\geq1$ with $m_{j}\geq1$ the denominator of $C_{\lambda_{j},0}^{(s)}$ is the same as that of $C_{\lambda,0}^{(s)}$, except that $j!$ now appears only to the power $m_{j}-1$ and we have $(m_{j}-1)!$ instead of $m_{j}!$, and if $j\geq2$ then $(j-1)!$ comes with the power $m_{j-1}+1$ and the denominator also contains $(m_{j-1}+1)!$. The multiplier $m_{j-1}+1$ is trivial when $j=1$ and cancels the latter factorial to $m_{j-1}!$ otherwise, and after we multiply the numerator and denominator by $jm_{j}$, we also get the required denominator $m_{j}!$, and the powers of $j!$, as well as of $(j-1)!=\frac{j!}{j}$ when $j\geq2$, become the correct ones as well. This shows that after multiplying by the denominator of the left hand side, the right hand side becomes \[\sum_{j\geq1}|\lambda_{j}|!jm_{j}\overline{\delta}_{m_{j},0}=(|\lambda|-1)!\sum_{j\geq1}jm_{j}=(|\lambda|-1)!\cdot|\lambda|=|\lambda|!,\] where the expression on the left is obtained via Lemma \ref{lambdaj} again, the first equality is based on the fact that $\overline{\delta}_{m_{j},0}$ vanishes only when the multiplier $m_{j}$ vanishes and can therefore be ignored, and we then use the definition of $|\lambda|$. As this is the numerator on the left hand side as well, this proves the lemma.
\end{proof}
The details of the proof of Lemma \ref{recr0} will be useful for the proof of the main result. As the latter proof will work inductively, we provide the full proof of Equation \eqref{FaadiBruno}, since we shall use these arguments as well.
\begin{prop}
The formula from Equation \eqref{FaadiBruno} is valid for every $n$, i.e., for any $n\geq0$ we have
\[(f\circ\varphi)^{(n)}=\sum_{\lambda \vdash n}C_{\lambda,0}^{(s)}\big(f^{(\ell(\lambda))}\circ\varphi\big)\prod_{i=1}^{n}\big(\varphi^{(i)}\big)^{m_{i}}.\] \label{FdBprop}
\end{prop}

\begin{proof}
We argue by induction on $n$, where the case with $n=0$, consisting only of the empty partition with $C_{\lambda}=1$, is a tautology. Assuming that $n>0$ and that the result holds for $n-1$, we have to differentiate the result for $n-1$ with respect to $t$ and show that it gives the asserted expression. Now, for any $\mu \vdash n-1$ we can first differentiate $f^{(\ell(\mu))}\circ\varphi$ and get $f^{(\ell(\mu)+1)}\circ\varphi$ times $\varphi'$, which (up to the coefficient) corresponds to the partition of $n$ obtained by adding another number $a_{\ell(\mu)+1}=1$ to $\mu$, and we write $j=1$ in this case. We can also differentiate one of the other multipliers, which we write as $\varphi^{(j-1)}$ for some $j\geq2$, and render it $\varphi^{(j)}$, yielding (again up to the coefficient) a term like the one associated with the partition of $n$ in which one of the $a_{l}$s of $\mu$ which equal $j-1$ was increased to $j$ (there are $m_{j-1}(\mu)$ such numbers, and indeed this operation comes with multiplicity $m_{j-1}(\mu)$ because of the power to which $\varphi^{(j-1)}$ appears in the expression from the induction hypothesis). In any case the resulting partition $\lambda$ satisfies $m_{j}(\lambda)\geq1$, and the contributing partition $\mu$ is $\lambda_{j}$ from Definition \ref{lambdajdef} (we write $\lambda_{1}$ also for $\lambda-\varepsilon_{1}$ in the case with $j=1$). Hence we indeed obtain a sum over $\lambda \vdash n$ of the required expressions, and we need to verify the coefficients. But given such $\lambda$, we get contributions exactly from those $\lambda_{j}$ for which $m_{j}(\lambda)\geq1$, and Lemma \ref{lambdaj} shows that the contribution to the coefficient is $C_{\lambda_{j},0}^{(s)}$ times $m_{j-1}(\lambda)+1$ (also for $j=1$, where the latter multiplier is indeed 1). Therefore the resulting coefficient is the one from the right hand side of Lemma \ref{recr0}, which is the desired one by this lemma. This proves the proposition.
\end{proof}

\section{Some Properties of the Functions $e_{r}$ \label{Propofer}}

The coefficients $C_{\lambda,r}^{(s)}$ do depend on $s$, and also involve the non-trivial elementary symmetric functions $\{e_{r}\}_{r\geq1}$ (the function $e_{0}$ appearing in $C_{\lambda,0}^{(s)}$ is just 1). We now present some of their properties that we shall need. Most of the material can be found in many places in the literature, e.g., Section 2 in Chapter 1 of \cite{[M]}, but we include it here for completeness and since it is short and simple.

\smallskip

We first consider what happens to the elementary symmetric function $e_{r}$, of the (finitely many) numbers $b_{l}$ with $1 \leq l \leq L$, say, when we subtract some number $c$ from one particular number $b_{l}$. This includes, as a special case, the formula for omitting $b_{l}$, where we simply take $c=b_{l}$.
\begin{lem}
For any $r\geq0$, write $e_{r}$ for $e_{r}(b_{1},\ldots,b_{L})$ for some numbers $b_{l}$, $1 \leq l \leq L$. Then replacing $b_{l}$ for one index $l$ by $b_{l}-c$ sends $e_{r}$ to the expression $e_{r}-c\sum_{k=1}^{r}(-b_{l})^{k-1}e_{r-k}$, which for $c=b_{l}$ becomes $\sum_{k=0}^{r}(-b_{l})^{k}e_{r-k}$. \label{ertrans}
\end{lem}

\begin{proof}
The formula expressing the coefficients of a polynomial using its roots transforms, by a simple operation, to the equality
\begin{equation}
\sum_{r=0}^{L}e_{r}X^{r}=\prod_{l=1}^{L}(1+b_{l}X). \label{prodfore}
\end{equation}
Our operation replaces $b_{l}$ by $b_{l}-c$, so that we need the coefficients of the power series in $X$ obtained by multiplying by $\frac{1+(b_{l}-c)X}{1+b_{l}X}$. Since after expanding the denominator geometrically this multiplier becomes $1-c\sum_{k=1}^{\infty}(-b_{l})^{k-1}X^{k}$, multiplying by the left hand side gives the series with the asserted coefficients. The case with $c=b_{l}$ is now immediate. This proves the lemma.
\end{proof}

Another well-known identity that we shall need is the following one.
\begin{lem}
With $e_{r}$ as in Lemma \ref{ertrans}, and with $p_{k}$ defined to be $p_{k}(b_{1},\ldots,b_{L})$ as well, the sum $\sum_{k=1}^{r}(-1)^{k-1}p_{k}e_{r-k}$ equals $re_{r}$ for every $r\geq1$. \label{epiden}
\end{lem}

\begin{proof}
Take the logarithm of Equation \eqref{prodfore}, and substitute the series for $\log(1-z)$ in the right hand side. The right hand side then becomes $\sum_{k=1}^{\infty}(-1)^{k-1}\frac{p_{k}}{k}X^{k}$, and after differentiating we get $\sum_{k=1}^{\infty}(-1)^{k-1}p_{k}X^{k-1}$. But differentiating the logarithm of the left hand side gives the quotient between $\sum_{r\geq1}re_{r}X^{r-1}$ and $\sum_{r\geq0}e_{r}X^{r}$, and after we multiply by the denominator, the result follows by comparing the coefficient of $X^{r-1}$ on both sides. This proves the lemma.
\end{proof}
We remark that the result of Lemma \ref{epiden} is trivially true also for $r=0$, but we shall need it only for $r\geq1$.

\smallskip

We can now establish the extension of Lemma \ref{recr0} to $r\geq1$.
\begin{prop}
For every $s\geq0$, $r\geq0$, and partition $\lambda$, we have the equality \[C_{\lambda,r}^{(s)}=\sum_{j\geq1}(m_{j-1}+1)\overline{\delta}_{m_{j},0}C_{\lambda_{j},r}^{(s)}+\overline{\delta}_{r,0}\overline{\delta}_{m_{s+1},0}C_{\lambda-\varepsilon_{s+1},r-1}^{(s)}.\] \label{recrpos}
\end{prop}

\begin{proof}
The case $r=0$ is simply the equality from Lemma \ref{recr0} (because $\overline{\delta}_{r,0}$ vanishes), so that we may assume $r\geq1$. The same argument from the proof of that lemma, and the fact that the numerators in all the terms on the right hand side involve $(n-1)!$ (by Lemma \ref{lambdaj}), reduce us to proving the equality
\begin{equation}
\sum_{j\geq1}jm_{j}e_{r}\big((\lambda_{j}^{>s})_{s}\big)+(s+1)!m_{s+1}e_{r-1}\big((\lambda^{>s}-\varepsilon_{s+1})_{s}\big)=ne_{r}\big((\lambda^{>s})_{s}\big) \label{sumnum}
\end{equation}
(because the denominator associated with $\lambda-\varepsilon_{s+1}$ misses one power of $(s+1)!$ and one coefficient of $m_{s+1}$ to become equal to that of $\lambda$). Now, it is easy to verify that the partition $\lambda_{j}^{>s}$ is the same as $\lambda^{>s}$ when for $j \leq s$, it coincides with $\lambda^{>s}-\varepsilon_{s+1}$ for $j=s+1$, and it equals $(\lambda^{>s})_{j}$ when $j>s+1$. It follows that $(\lambda_{j}^{>s})_{s}$ is $(\lambda^{>s})_{s}$ in the first case, $(\lambda^{>s})_{s}-\varepsilon_{(s+1)_{s}}$ in the second one, and it obtained from $(\lambda^{>s})_{s}$ by replacing one instance of $(j)_{s}$ by $(j-1)_{s}$.

We may therefore evaluate the summands with $j>s$ on the left hand side of Equation \eqref{sumnum} via Lemma \ref{ertrans}. For $j=s+1$ we obtain
\[(s+1)m_{s+1}\big[e_{r}\big((\lambda^{>s})_{s}\big)-(s+1)_{s}\sum_{k=1}^{r}\big(-(s+1)_{s}\big)^{k-1}e_{r-k}\big((\lambda^{>s})_{s}\big)\big],\] and as the remaining term becomes \[(s+1)!m_{s+1}\sum_{k=1}^{r}\big(-(s+1)_{s}\big)^{k-1}e_{r-k}\big((\lambda^{>s})_{s}\big)\] after a summation index change, where $(s+1)!=(s+1)_{s}=(s+1)(s)_{s-1}$, the missing factor of $s+1$ in the latter equation in comparison with the preceding one shows that they combine to \[(s+1)m_{s+1}\big[e_{r}\big((\lambda^{>s})_{s}\big)-s(s)_{s-1}\sum_{k=1}^{r}\big(-(s+1)_{s}\big)^{k-1}e_{r-k}\big((\lambda^{>s})_{s}\big)\big]\] (this is valid also when $s=0$, where the sum over $k$ is indeed multiplied by a vanishing coefficient). On the other hand, when $j>s+1$ we observe that $(j)_{s}-(j-1)_{s}=s(j-1)_{s-1}$ (once again this equality holds also when $s=0$, in the form of $1-1=0$), so that we get \[jm_{j}\big[e_{r}\big((\lambda^{>s})_{s}\big)-s(j-1)_{s-1}\sum_{k=1}^{r}\big(-(j)_{s}\big)^{k-1}e_{r-k}\big((\lambda^{>s})_{s}\big)\big],\] of which the latter expression is the case with $j=s+1$. Altogether the left hand side of Equation \eqref{sumnum} equals
\begin{equation}
\sum_{j\geq1}jm_{j}e_{r}\big((\lambda^{>s})_{s}\big)-\sum_{j>s}sjm_{j}(j-1)_{s-1}\sum_{k=1}^{r}\big(-(j)_{s}\big)^{k-1}e_{r-k}\big((\lambda^{>s})_{s}\big). \label{LHS}
\end{equation}

But recalling that $j(j-1)_{s-1}=(j)_{s}$ for every such $j$, we deduce that for each $1 \leq k \leq r$, the summand $e_{r-k}\big((\lambda^{>s})_{s}\big)$ (which is independent of $j$) is multiplied by $-s\sum_{j>s}(-1)^{k-1}m_{j}(j)_{s}^{k}$, which equals $-s(-1)^{k-1}p_{k}\big((\lambda^{>s})_{s}$ by definition. It thus follows from Lemma \ref{epiden} that the second expression in Equation \eqref{LHS} is just $-sre_{r}\big((\lambda^{>s})_{s}\big)$. Since the first term there is $|\lambda|e_{r}\big((\lambda^{>s})_{s}\big)$, and we know that $|\lambda|=n-rs$, we indeed establish Equation \eqref{sumnum}. This proves the proposition.
\end{proof}

\section{The Main Result and Some Consequences \label{Main}}

We can now prove our main result.
\begin{thm}
For every three functions with derivatives of high enough order, and for every integer $s\geq0$, The $n$th derivative of $(f\circ\varphi)\cdot\big(g\circ\varphi^{(s)}\big)$ is
\[\sum_{r\geq0}\sum_{\lambda \vdash n+rs}\frac{n!e_{r}\big((\lambda^{>s})_{s}\big)}{\prod_{i=1}^{n}(i!)^{m_{i}}m_{i}!}\big(f^{(\ell(\lambda)-r)}\circ\varphi\big)\big(g^{(r)}\circ\varphi^{(s)}\big)\prod_{i=1}^{n}\big(\varphi^{(i)}\big)^{m_{i}}.\] Moreover, this expression coincides with the one from Equation \eqref{finres}, in which we pose the restrictions $r \leq n$ and $\ell(\lambda^{>s}) \geq r$, and is therefore finite. \label{main}
\end{thm}

\begin{proof}
We first prove that the two restrictions in Equation \eqref{finres} are redundant. Indeed, the elementary symmetric function $e_{r}$ is known to vanish when we substitute less than $r$ distinct parameters, and as $\ell\big((\lambda^{>s})_{s}\big)=\ell(\lambda^{>s})$, adding the restriction that this length is at least $r$ does not affect the resulting sum. Now, if $\lambda \vdash n+rs$ and $\ell(\lambda^{>s}) \geq r$ then $\lambda$ contains at least $r$ summands, all of which are at least $s+1$. We therefore obtain the inequality $n+rs=|\lambda| \geq r(s+1)$, which implies that such partitions exist only when $r \leq n$ as desired. The finiteness of the set of possible indices $r$ and of the number of partitions of $n+rs$ for every $0 \leq r \leq n$ thus yield the finiteness of our formula.

For establishing the formula itself we follow the proof of Proposition \ref{FdBprop} and argue by induction on $n$, where the case with $n=0$ is now clearly trivial (we only have $r=0$ and $\lambda\vdash0$). Assume now that $n>0$ and that our formula is true for $n-1$, and differentiate with respect to $t$ again. Given $0 \leq k \leq n-1$ and $\mu \vdash n-1+ks$, we first obtain contributions from differentiating $f^{(\ell(\mu)-k)}\circ\varphi$, yielding a summand associated with $r=k$ and with a partition $\lambda \vdash n+ks$, with $m_{1}(\lambda)\geq1$ and such that $\lambda_{1}=\lambda-\varepsilon_{1}=\mu$ via Definition \ref{lambdajdef}. The differentiation of one of the multipliers $\varphi^{(j-1)}$ with $j\geq2$ will again produce a summand corresponding to $r=k$ and to a partition $\lambda \vdash n+ks$ for which $m_{j}(\lambda)\geq1$ and $\lambda_{j}=\mu$ as in Definition \ref{lambdajdef}, with an extra multiplier of $m_{j-1}(\mu)$. But here we can also differentiate the multiplier $g^{(k)}\circ\varphi^{(s)}$, whose derivative is $\big(g^{(k+1)}\circ\varphi^{(s)}\big)\cdot\varphi^{(s+1)}$, and the resulting summand is based on $r=k+1$ and on $\lambda \vdash n+(k+1)s$, where here $m_{s+1}(\lambda)\geq1$ and $\mu$ is $\lambda-\varepsilon_{s+1}$. Therefore we indeed obtain the asserted sum over $r$ and $\lambda$, and it remain to compare the coefficients. We take some $r\geq0$ and $\lambda \vdash n+rs$, and using Lemma \ref{lambdaj} we deduce, from the induction hypothesis, that we have a contribution of $m_{j-1}(\lambda)+1$ times $C_{\lambda_{j},r}^{(s)}$ for every $j\geq1$ such that $m_{j}(\lambda)\geq1$, and when $r\geq1$ and $m_{s+1}(\lambda)\geq1$ we also obtain a contribution of $C_{\lambda-\varepsilon_{s+1},r-1}^{(s)}$, with no extra coefficient. In other words, the total coefficient multiplying the summand associated with $r$ and $\lambda$ is the one appearing in the right hand side of Proposition \ref{recrpos}, which therefore equals $C_{\lambda,r}^{(s)}$ by this proposition. This proves the theorem.
\end{proof}
We remark that the restriction $r \leq n$ in Theorem \ref{main} and Equation \eqref{finres} corresponds to the fact that in such an $n$th derivative we cannot differentiate $g$ to an order exceeding $n$.

\smallskip

We recall that the coefficients $C_{\lambda,0}^{(s)}$ from Equation \eqref{FaadiBruno} and Proposition \eqref{FdBprop} are integers. Indeed, each such coefficient $C_{\lambda,0}^{(s)}$ has a combinatorial meaning, where it counts the number of ways to put $n=|\lambda|$ numbered balls in boxes whose sizes are determined by $\lambda$, where boxes of the same size are identical. The first consequence that we draw from Theorem \ref{main} is the integrality of the other coefficients, which is much less trivial in first sight.
\begin{cor}
For every $r\geq0$, $s\geq0$, $n\geq0$, and partition $\lambda \vdash n+rs$, the rational number $C_{\lambda,r}^{(s)}$ from Equation \eqref{finres} and Theorem \ref{main} is an integer. \label{integer}
\end{cor}

\begin{proof}
Theorem \ref{main} shows that for $n=0$ this coefficient is 1 when $r=0$ and 0 otherwise, and Proposition \ref{recrpos} evaluates each such coefficient as a combination of previous ones with integral coefficients. The assertion thus follows by induction on $n$ as in the proof of Theorem \ref{main}. This proves the corollary.
\end{proof}
We remark that Corollary \ref{integer} does not follow from the case $r=0$ by the obvious integrality of $e_{r}\big((\lambda^{>s})_{s}\big)$, because for $\lambda \vdash n+rs$ the coefficient $C_{\lambda,0}^{(s)}$ has $(n+rs)!$ in the numerator, while for $C_{\lambda,r}^{(s)}$ it is just $n!$. Note that the case $s=1$ in Corollary \ref{integer} involves the moments of the partition $\lambda$ itself (up to the truncation to $\lambda^{>1}$), because the operation of taking $a$ to $(a)_{1}$ is trivial. In particular we obtain the integrality of $n!e_{r}(\lambda^{>1})\big/\prod_{j}j!^{m_{j}}m_{j}!$ for every partition $\lambda \vdash n+r$.

\smallskip

Recall that Theorem \ref{main} evaluates the $n$th derivative of a product of compositions, which we can also evaluate using Leibnitz's Rule and the original formula of Fa\`{a} di Bruno. We now use this fact for obtaining an alternative proof for that theorem, which will also make the connection to the algebraic formulae from \cite{[MS]} more visible. This proof will use an identity for the elementary moments $e_{r}\big((\lambda^{>s})_{s}\big)$, for describing which we recall the following definition.
\begin{defn}
Let $\mu$ and $\nu$ be two partitions. Then $\mu\cup\nu$ is the partition obtained by taking all the summands in $\mu$ and all those of $\nu$, combining them together, and ordering the resulting sequence in decreasing order. In addition, if $s\geq0$ is an integer then we denote by $\mu^{+s}$ the partition obtained by adding $s$ to each of the summands of $\mu$. \label{union}
\end{defn}
It is clear from Definition \ref{union} that
\begin{equation}
|\mu\cup\nu|=|\mu|+|\nu|,\ \ell(\mu\cup\nu)=\ell(\mu)+\ell(\nu),\mathrm{\ and\ }m_{i}(\mu\cup\nu)=m_{i}(\mu)+m_{i}(\nu)\mathrm{\ for\ }i\geq1 \label{unsum}
\end{equation}
as well as
\begin{equation}
\ell(\mu^{+s})=\ell(\mu),\ |\mu^{+s}|=|\mu|+s\ell(\mu),\mathrm{\ and\ }m_{i}(\mu^{+s})=m_{i-s}(\mu)\mathrm{\ for all\ }i>s \label{val+s}
\end{equation}
(while $m_{i}(\mu^{+s})=0$ in case $i \leq s$). Equation \eqref{unsum} now shows that given two partitions $\lambda$ and $\mu$, we can write $\lambda$ as $\mu\cup\nu$ for some partition $\nu$ if and only if $m_{i}(\lambda) \geq m_{i}(\mu)$ for every $i\geq1$ (a condition which we write as simply $\mu\leq\lambda$), and then the partition $\nu$ for which $\lambda=\mu\cup\nu$ is uniquely determined. Similarly, a partition $\lambda$ is of the form $\mu^{+s}$ if and only if $m_{i}(\lambda)=0$ in case $i \leq s$ (i.e., $\lambda=\lambda^{>s}$), and then $\mu$ is uniquely determined.

The formula that we shall use is the following one.
\begin{prop}
Given a partition $\eta$ and integers $s\geq0$ and $r\geq0$, we have \[e_{r}\big((\eta)_{s}\big)=\sum_{\ell(\nu)=r}\prod_{i\geq1}\binom{m_{i}(\eta)}{m_{i}(\nu)}\prod_{i\geq1}(i)_{s}^{m_{i}(\nu)}.\] \label{formfore}
\end{prop}

\begin{proof}
The elementary moment $e_{r}\big((\eta)_{s}\big)$ is, by definition, the sum of all the possible products of $(j)_{s}$ for $r$ different terms $j$ that appear in $\eta$. Any such collection corresponds to a partition $\nu$ that satisfies $\nu\leq\eta$ and $\ell(\nu)=r$, but it may happen that different choices will produce the same partition $\nu$. More precisely, for a choice of $r$ entries from $\eta$ to produce $\nu$ we need to choose $m_{i}(\nu)$ of the entries that equal $i$ out of the $m_{i}(\eta)$ possible ones, so that doing it for every $i$ yields the asserted product as the multiplicity (which also makes the condition $\nu\leq\eta$ redundant, since the product vanishes wherever it is not satisfied). Since for any $\nu$ the expression appearing with that multiplicity is $\prod_{i\geq1}(i)_{s}^{m_{i}(\nu)}$, we indeed obtain the asserted result. This proves the proposition.
\end{proof}

\begin{cor}
For a partition $\lambda$ with $r$ and $s$ as above we get \[e_{r}\big((\lambda^{>s})_{s}\big)=\sum_{\ell(\mu)=r}\prod_{i>s}\binom{m_{i}(\lambda)}{m_{i-s}(\mu)}\prod_{i\geq1}\bigg(\frac{(i+s)!}{i!}\bigg)^{m_{i}(\mu)}.\] \label{combiden}
\end{cor}

\begin{proof}
We consider the formula from Proposition \ref{formfore}. For every partition $\nu$ satisfying $\nu\leq\lambda^{>s}$ we have $m_{i}(\nu)=0$ for any $i \leq s$, which allows us to write $\nu$ as $\mu^{+s}$ for a partition $\mu$. Equation \eqref{val+s} and changing the index $i$ to $i-s$ thus transforms the formula from that proposition into the asserted one. This proves the corollary.
\end{proof}

We remark about the case $s=0$ in Proposition \ref{formfore} and Corollary \ref{combiden}, where both formulae coincide since $\lambda^{>0}=\lambda$. Since $(a)_{0}=1$ for every $a\geq1$ (and $\frac{(i+0)!}{i!}=1$), both equalities reduce to
\begin{equation}
\binom{\ell(\lambda)}{r}=\sum_{\ell(\mu)=r}\prod_{i\geq1}\binom{m_{i}(\lambda)}{m_{i}(\mu)}. \label{binellr}
\end{equation}
While Equation \eqref{binellr} seems non-trivial algebraically, the proof of Proposition \ref{formfore} explains its combinatorial interpretation. Indeed, one may view the partition $\lambda$ as a marking of $\ell(\lambda)$ balls, where the number of balls that are marked by $i$ is $m_{i}(\lambda)$, and we ask in how many ways can one choose $r$ of these balls. Now, every such choice will correspond to a partition $\mu$ with $\ell(\mu)=r$ (and $m_{i}(\mu) \leq m_{i}(\lambda)$ for every $i\geq1$) according to the markings, and given such a partition $\mu$, the number of options to choose $m_{i}(\mu)$ balls out of the $m_{i}(\lambda)$ ones that are marked with $i$ is $\binom{m_{i}(\lambda)}{m_{i}(\mu)}$. The answer to our question, which is known to be $\binom{\ell(\lambda)}{r}$, is thus obtained by multiplying over $i$ for each $\mu$, and then summing over $\mu$ as desired. Therefore Proposition \ref{formfore} may be viewed as a generalization of Equation \eqref{binellr} to $s\geq0$, and Corollary \ref{combiden} is its special case that is related to Theorem \ref{main}.

\smallskip

For obtaining the alternative proof of Theorem \ref{main} (via Corollary \ref{combiden}), we shall use the following lemma.
\begin{lem}
For four sufficiently differentiable functions $f$, $g$, $\varphi$, and $\psi$, the $n$th derivative of $(f\circ\varphi)\cdot(g\circ\psi)$ is \[\sum_{r=0}^{n}\sum_{\rho \vdash n}\!\frac{n!\big(f^{(\ell(\rho)-r)}\circ\varphi\big)\big(g^{(r)}\circ\psi\big)}{\prod_{i\geq1}i!^{m_{j}(\rho)}m_{i}(\rho)!}\!\!\sum_{\ell(\mu)=r}\prod_{i\geq1}\!\binom{m_{i}(\rho)}{m_{i}(\mu)} \big(\varphi^{(i)}\big)^{m_{i}(\rho)-m_{i}(\mu)}\!\big(\psi^{(i)}\big)^{m_{i}(\mu)}\!.\] \label{derofprod}
\end{lem}

\begin{proof}
Leibnitz' Rule and Equation \eqref{FaadiBruno} (or Proposition \ref{FdBprop}) evaluate the derivative in question as $n!$ times \[\sum_{p=0}^{n}\sum_{\nu \vdash p}\frac{\big(f^{(\ell(\nu))}\circ\varphi\big)\prod_{i\geq1}\big(\varphi^{(i)}\big)^{m_{i}(\nu)}}{\prod_{i\geq1}i!^{m_{i}(\nu)}m_{i}(\nu)!}\sum_{\mu \vdash n-p}\frac{\big(g^{(\ell(\mu))}\circ\psi\big)\prod_{i\geq1}\big(\psi^{(i)}\big)^{m_{i}(\mu)}}{\prod_{i\geq1}i!^{m_{i}(\mu)}m_{i}(\mu)!},\] where the denominators of $\binom{n}{p}$ cancel with the numerators of $C_{\nu,0}^{(s)}$ and $C_{\mu,0}^{(s)}$ from that equation. We separate the second sum according to $r=\ell(\mu)$, and we replace $\nu$ by $\rho=|\mu\cup\nu|$, for which $|\rho|=n$ by Equation \eqref{unsum}, which also presents the terms associated with $\mu$ and $\nu$ in the asserted form (using $\rho$ and $\mu$). Since then $\mu\leq\rho$ and $p$ is determined by $\mu$ as $n-|\mu|$, and the product of the binomial coefficients makes the condition $\mu\leq\rho$ redundant, we indeed obtain the stated formula. This proves the lemma.
\end{proof}
For obtaining Theorem \ref{main}, we do not employ Lemma \ref{derofprod} (with $\psi=\varphi^{(s)}$, of course) as stated, but use the argument from its proof. Given $r$ and $\mu$, instead of using $\rho=\nu\cup\mu$, we set $\lambda=\nu\cup\mu^{+s}$, and we write $\prod_{i\geq1}\frac{1}{i!^{m_{i}(\mu)}}$ as $\prod_{i>s}\frac{1}{i!^{m_{i}(\mu^{+s})}}\prod_{i\geq1}\big(\frac{(i+s)!}{i!}\big)^{m_{i}(\mu)}$, using Equation \eqref{val+s} for the first product. Equation \eqref{unsum} now implies that $|\lambda|=n+rs$ and $\ell(\nu)$ is once again $\ell(\lambda)-r$ when $\ell(\mu)=r$, and gathering the derivatives of $\varphi$ and $\psi=\varphi^{(s)}$ yields just $\prod_{i\geq1}\big(\varphi^{(i)}\big)^{m_{i}(\lambda)}$. Since the parts depending on $\mu$ are just those from Corollary \ref{combiden}, the formula from Theorem \ref{main} is established.

The case where Lemma \ref{derofprod} can be applied directly is when $s=0$, so that $\rho=\lambda$ and $\psi=\varphi$, and the product on the right hand side of that lemma involves just the binomial coefficients and $\prod_{i=1}^{n}\big(\varphi^{(i)}\big)^{m_{i}(\lambda)}$. Since in this case all the partitions in Theorem \ref{main} are of $n$, the only dependence on $r$ goes into the sum $\sum_{r=0}^{n}\binom{\ell(\lambda)}{r}\big(f^{(\ell(\lambda)-r)}\circ\varphi\big)\big(g^{(r)}\circ\varphi\big)$, where the sum is essentially up to $\ell(\lambda)$ because of the binomial coefficient. As Leibnitz's Rule implies that this is simply $(fg)^{\ell(\lambda)}\circ\varphi$, the total formula is indeed, via Equation \eqref{FaadiBruno}, the $n$th derivative of $(fg)\circ\varphi$, a function which is clearly the same as $(f\circ\varphi)\cdot\big(g\circ\varphi\big)$.

\section{Modified Bell Polynomials \label{Bell}}

In order to connect our results with the algebraic setting given in, e.g., Subsection 3.2.5 of \cite{[MS]}, we recall that Fa\`{a} di Bruno's formula from Equation \eqref{FaadiBruno} can be expressed as
\begin{equation}
(f\circ\varphi)^{(n)}(t)=\sum_{k=0}^{n}\mathcal{B}_{n,k}\big(\varphi'(t),\ldots,\varphi^{(n-k+1)}(t)\big)f^{(k)}\big(\varphi(t)\big) \label{FdBBell}
\end{equation}
using the \emph{partial Bell polynomials}
\begin{equation}
\mathcal{B}_{n,k}(y_{1},\ldots,y_{n-k+1})=\!\sum_{\substack{\lambda \vdash n \\ \ell(\lambda)=k}}C_{\lambda,0}^{(0)}\prod_{i=1}^{n-k+1}y_{i}^{m_{i}}=\sum_{\substack{\lambda \vdash n \\ \ell(\lambda)=k}}\frac{n!}{\prod_{i=1}^{n}(i!)^{m_{i}}m_{i}!}\prod_{i=1}^{n-k+1}y_{i}^{m_{i}}. \label{Bellpols}
\end{equation}
The polynomial $\mathcal{B}_{n,k}$ has degree $k$ and weighted degree $n$, where the weighted degree is according to each variable $y_{i}$ having weight $i$ (with the convention that $\mathcal{B}_{0,0}$ is defined to be 1). The \emph{(complete) Bell polynomials} are the sums $\mathcal{B}_{n}(y_{1},\ldots,y_{n})=\sum_{k=1}^{n}\mathcal{B}_{n,k}(y_{1},\ldots,y_{n-k+1})$ (the sum is taken from 1 if $n\geq1$, but the index $k=0$ is included because $\mathcal{B}_{n}$ is defined to be 1 as well), where $\mathcal{B}_{n,k}$ can be recognized as the part of $\mathcal{B}_{n}$ that is homogenous of degree $k$. Some of their relations with other combinatorial numbers and functions can be seen by setting the variables $y_{i}$ to lie in a geometric sequence, i.e., $y_{i}=c^{i}x$ with $c$ and $x$ constant (this is equivalent to taking $\varphi(t)=e^{ct}$, where $x=e^{t}$). Then $\mathcal{B}_{n,k}(y_{1},\ldots,y_{n-k+1})$ becomes just $S(n,k)c^{n}x^{k}$, where $S(n,k)$ is the Stirling number of the second kind, and $\mathcal{B}_{n}(y_{1},\ldots,y_{n})$ reduces to $c^{n}$ times the Touchard polynomial $T_{n}(x)=\sum_{k=0}^{n}S(n,k)x^{k}$.

\smallskip

We would like to define polynomials of Bell type that will correspond to the formula from Theorem \ref{main} by separating according to $k=\ell(\lambda)$, for which we shall first need to bound the indices in terms of $n$. We have already seen that the index $r$ must lie between 0 and $n$, and we get $k=\ell(\lambda)\geq\ell(\lambda^{>s}) \geq r$ for non-triviality. On the other hand, if $\lambda \vdash n+rs$ satisfies $\ell(\lambda^{>s}) \geq r$ and $\ell(\lambda)=k$ then we can express $n+rs=|\lambda|$ as \[\sum_{i\geq1}im_{i}=\sum_{i>s}im_{i}+\sum_{i=1}^{s}im_{i}\geq(s+1)\sum_{i>s}m_{i}+\sum_{i=1}^{s}m_{i}=s\ell(\lambda^{>s})+\ell(\lambda) \geq sr+k,\] implying the inequality $k \leq n$ as well. We therefore set, for every $0 \leq r \leq k \leq n$ and every $s\geq0$, the \emph{modified partial Bell polynomial} \[\widetilde{\mathcal{B}}_{n,k,r}^{(s)}(y_{1},\ldots,y_{s+n-k+1})=\sum_{\substack{\lambda \vdash n+rs \\ \ell(\lambda)=k \\ \ell(\lambda^{>s}) \geq r}}\frac{n!e_{r}\big((\lambda^{>s})_{s}\big)}{\prod_{i=1}i!^{m_{i}}m_{i}!}\prod_{i=1}^{s+n-k+1}y_{i}^{m_{i}}\] (where the summand associated with $\lambda$ can be written as $C_{\lambda,r}^{(s)}\prod_{i=1}^{s+n-k+1}y_{i}^{m_{i}}$), and by what we have seen we may omit the inequalities between $r$, $k$, and $n$ and obtain just 0 wherever they are not satisfied (as well as when $k=0$ and $n>0$). The degree of $\widetilde{\mathcal{B}}_{n,k,r}^{(s)}$ is $k$, and the weighted degree is $n+rs$, implying that there is a \emph{modified Stirling number of the second kind} $\widetilde{S}(n,k,r)$ such that with $y_{i}=c^{i}x$ (i.e., $\varphi(t)=e^{ct}$ and $x=e^{t}$ again) this polynomial reduces to $c^{n+rs}\widetilde{S}(n,k,r)x^{k}$ (we shall soon see that they are indeed independent of $s$). With these polynomials, Theorem \ref{main} takes the form
\begin{equation}
\big[(f\circ\varphi)\cdot\big(g\circ\varphi^{(s)}\big)\big]^{(n)}=\sum_{r=0}^{n}\sum_{k=0}^{n}\widetilde{\mathcal{B}}_{n,k,r}^{(s)}\big(\varphi',\ldots,\varphi^{(s+n-k+1)}\big)\big(f^{(k-r)}\circ\varphi\big) \big(g^{(r)}\circ\varphi^{(s)}\big). \label{modBell}
\end{equation}
The \emph{modified complete Bell polynomial} $\widetilde{\mathcal{B}}_{n}^{(s)}(y_{1},\ldots,y_{s+n})$ is similarly defined to be $\sum_{r=0}^{n}\sum_{k=r}^{n}\widetilde{\mathcal{B}}_{n,k,r}^{(s)}(y_{1},\ldots,y_{s+n-k+1})$, where now the part that is homogenous of degree $k$ is the partial sum $\sum_{r=0}^{n}\widetilde{\mathcal{B}}_{n,k,r}^{(s)}(y_{1},\ldots,y_{s+n-k+1})$. Note that the part with $r=0$ in Equation \eqref{modBell} consists of the terms in which the differentiation is not applied to $g\circ\varphi^{(s)}$, implying that $\widetilde{\mathcal{B}}_{n,k,0}^{(s)}$ is just the ordinary Bell polynomial $\mathcal{B}_{n,k}$ and thus $\widetilde{S}(n,k,0)=S(n,k)$. We remark that if $s>n+2-k$ then $\widetilde{\mathcal{B}}_{n,k,r}^{(s)}(y_{1},\ldots,y_{s+n-k+1})$ is independent of some of the variables, namely $y_{i}$ with $n+1-k<i \leq s$ do not appear in it---see Proposition \ref{mBexp} below.

\smallskip

It would be an interesting question to examine which kind of properties do the modified partial and complete Bell polynomial $\widetilde{\mathcal{B}}_{n,k,r}^{(s)}$ have. As one example we consider the formula \[\mathcal{B}_{n+1,k+1}(y_{1},\ldots,y_{n-k+1})=\sum_{l=0}^{n-k}\binom{n}{l}y_{l+1}\mathcal{B}_{n-l,k}(y_{1},\ldots,y_{n-l-k+1}),\] which is the part of Equation (3.25) of \cite{[MS]} that is homogenous of degree $k+1$ (the formula itself is the same but with the complete Bell polynomials $\mathcal{B}_{n+1}$ and $\mathcal{B}_{n-l}$). It can be easily proved by comparing the formula for the derivative of order $n+1$ of $f\circ\varphi$ given in Equation \eqref{FdBBell} with the one arising from the $n$th derivative of $(f'\circ\varphi)\varphi'$ via Leibniz' Rule and Equation \eqref{FdBBell}, using the independence of the derivatives of $\varphi$ for arbitrary $\varphi$, and then taking the coefficient of $f^{(k+1)}\circ\varphi$. The related equality for the Stirling numbers of the second kind, obtained by the substitution $y_{i}=c^{i}x$, expresses $S(n+1,k+1)$ as $\sum_{l=0}^{n}\binom{n}{l}S(n-l,k)$. Applying the same argument for the derivatives of $(f\circ\varphi)\cdot\big(g\circ\varphi^{(s)}\big)$, using its first derivative $(f'\circ\varphi)\varphi'\cdot\big(g\circ\varphi^{(s)}\big)+(f\circ\varphi)\big(g'\circ\varphi^{(s)}\big)\varphi^{(s+1)}$, produces the equality \[\widetilde{\mathcal{B}}_{n+1,k+1,r}^{(s)}(\vec{y})=\sum_{l=0}^{n-k}\binom{n}{l}y_{l+1}\widetilde{\mathcal{B}}_{n-l,k,r}^{(s)}(\vec{y})+\sum_{l=0}^{n-k}\binom{n}{l}y_{l+s+1}\widetilde{\mathcal{B}}_{n-l,k,r-1}^{(s)}(\vec{y}),\] and after substituting $y_{i}=c^{i}x$ we find that the modified Stirling number of the second kind $\widetilde{S}(n+1,k+1,r)$ is $\sum_{l=0}^{n-k}\binom{n}{l}\big(\widetilde{S}(n,k,r)+\widetilde{S}(n,k,r-1)\big)$. Since this recurrence is independent of $s$, so are these modified Stirling numbers.

\smallskip

Recall that the proof of Theorem \ref{main} via Lemma \ref{derofprod} and Corollary \ref{combiden} used Leibniz' Rule and the classical formula of Fa\`{a} di Bruno, and note that applying Equation \eqref{FdBBell} in this argument expresses the coefficient of $\big(f^{(k-r)}\circ\varphi\big)\big(g^{(r)}\circ\psi\big)$ in the $n$th derivative of $(f\circ\varphi)\cdot(g\circ\psi)$ for $0 \leq r \leq k \leq n$ as
\begin{equation}
\sum_{p=r}^{n-k+r}\binom{n}{p}\mathcal{B}_{n-p,k-r}\big(\varphi',\ldots,\varphi^{(n-p-k+r+1)}\big)\mathcal{B}_{p,r}\big(\psi',\ldots,\psi^{(p-r+1)}\big). \label{LeibFdB}
\end{equation}
This yields the following result.
\begin{prop}
The modified partial Bell polynomial $\widetilde{\mathcal{B}}_{n,k,r}^{(s)}(y_{1},\ldots,y_{s+n-k+1})$ equals \[\sum_{p=r}^{n-k+r}\binom{n}{p}\mathcal{B}_{n-p,k-r}(y_{1},\ldots,y_{n-p-k+r+1})\mathcal{B}_{p,r}(y_{s+1},\ldots,y_{s+p-r+1}).\] The modified complete Bell polynomial $\widetilde{\mathcal{B}}_{n}^{(s)}(y_{1},\ldots,y_{s+n})$ is \[\sum_{p=0}^{n}\binom{n}{p}\mathcal{B}_{n-p}(y_{1},\ldots,y_{n-p+1})\mathcal{B}_{p}(y_{s+1},\ldots,y_{s+p+1}).\] \label{mBexp}
\end{prop}

\begin{proof}
The first formula is obtained by substituting $\psi=\varphi^{(s)}$ in Equation \eqref{LeibFdB} and comparing with the corresponding term in Equation \eqref{modBell}, and then using the independence of the derivatives of $\varphi$ for arbitrary $\varphi$ as above. The second one follows from the first by taking the sum over $r$ and $k$. This proves the proposition.
\end{proof}
It directly follows from Proposition \ref{mBexp} that the part of $\widetilde{\mathcal{B}}_{n}^{(s)}$ that is homogenous of degree $k$ is the sum over $r$ of the first expression there.

We deduce an explicit expression for the modified Stirling numbers of the second kind as well.
\begin{cor}
We have the equality $\widetilde{S}(n,k,r)=\sum_{p=r}^{n-k+r}S(n-p,k-r)S(p,r)$ for every $n$, $k$, and $r$. \label{Stirling}
\end{cor}
Note that the bounds $0 \leq r \leq k \leq n$ are imposed for non-vanishing by the non-triviality of the sum and the vanishing of the classical Stirling number $S(a,b)$ of the second kind unless $b\geq0$.
\begin{proof}
We substitute $y_{i}=c^{i}x$ in the first equality from Proposition \ref{mBexp}, where the left hand side becomes $\widetilde{S}(n,k,r)c^{n+rs}x^{k}$ by definition. On the right hand side the first multiplier is $S(n-p,k-r)c^{n-p}x^{k-r}$, and since $y_{s+i}=c^{s+i}x$ for every $i$ the second multiplier is $S(p,r)c^{p}(c^{s}x)^{r}$. As multiplying and summing over $p$ indeed yields the asserted right hand side times $c^{n+rs}x^{k}$, the equality follows. This proves the corollary.
\end{proof}
The substitution from the proof of Corollary \ref{Stirling} in the second equality from Proposition \ref{mBexp} should give a relation to the Touchard polynomials. Indeed, then the left hand side becomes $\sum_{r=0}^{n}\sum_{k=r}^{n}\widetilde{S}(n,k,r)c^{n+rs}x^{k}$, while the relation between the classical Bell and Touchard polynomials expresses the right hand side as $\sum_{p=0}^{n}\binom{n}{p}c^{n-p}T_{n-p}(x)c^{p}T_{p}(c^{s}x)$. Since the Touchard polynomials are of binomial type (this is the content of part $(6)$ of Theorem 3.29 of \cite{[MS]}), the latter expression is just $c^{n}T_{n}\big((c^{s}+1)x\big)$. By taking $c=1$ or $s=0$ we end up with the equality $\sum_{k=0}^{n}\big(\sum_{r=0}^{k}\widetilde{S}(n,k,r)\big)x^{k}=T_{n}(2x)$, so that the internal sum over $r$ is $2^{k}S(n,k)$. It will be interesting to see what other properties the modified Stirling numbers of the second kind and bell polynomials possess.

\medskip

\noindent\textsc{Einstein Institute of Mathematics, the Hebrew University of Jerusalem, Edmund Safra Campus, Jerusalem 91904, Israel}

\noindent E-mail address: zemels@math.huji.ac.il

\end{document}